\documentclass{article}
\usepackage{graphicx}
\usepackage{pgfplots}
\usepackage{subfigure}
\usepackage{caption}
\usepackage[justification=centering]{caption}
\usepackage{amsmath}
\usepackage{amssymb}
\usepackage{amsthm}
\usepackage{listings}
\usepackage{epstopdf}
\usetikzlibrary{arrows}
\usepackage{tikz}
\usepackage{float}
\usepackage[linesnumbered,ruled]{algorithm2e}
\usetikzlibrary{shapes.geometric, arrows}
\usepackage[utf8]{inputenc}
\usepackage{enumitem}
\usepackage[top=2.5cm, bottom=2.5cm, left=3cm, right=2.5cm]{geometry}
\usepackage{multicol}
\newtheorem{example}{Example}
\newtheorem{theorem}{Theorem}

\usepackage[titletoc,title]{appendix}
\usepackage{array}
\usepackage{authblk}

\usepackage[compress]{cite}

\begin{document}

	\title{Abstract Similarity, Fractals and Chaos \\ \vspace{0.7cm} \hspace{10.5cm} \footnotesize \textit{The art of doing mathematics consists in} \\ \hspace{10.6cm} \textit{finding that special case which contains} \\ \hspace{8.82cm}  \textit{all the germs of generality.} \\ \vspace{0.2cm} \hspace{14.2cm} \textit{David Hilbert}}	
	\author[]{Marat Akhmet\thanks{Corresponding Author Tel.: +90 312 210 5355, Fax: +90 312 210 2972, E-mail: marat@metu.edu.tr} }
	\author[]{Ejaily Milad Alejaily}
	\affil[]{Department of Mathematics, Middle East Technical University, 06800 Ankara, Turkey}

	\date{}
	\maketitle
	
	To prove presence of chaos for fractals, a new mathematical concept of abstract similarity is introduced. As an example, the space of symbolic strings on a finite number of symbols is proved to possess the property. Moreover, Sierpinski fractals, Koch curve as well as Cantor set satisfy the definition. A similarity map is introduced and the problem of chaos presence for the sets is solved by considering the dynamics of the map. This is true for Poincar\'{e}, Li-Yorke and Devaney chaos, even in multi-dimensional cases. Original numerical simulations which illustrate the results are delivered.
	
	\section{Introduction}
	
	Self-similarity is one of the most important concepts in modern science. From the geometrical point of view, it is defined as the property of objects whose parts, at all scales, are similar to the whole. Dealing with self-similarity goes back to the 17th Century when Gottfried Leibniz introduced the notions of recursive self-similarity \cite{Zmeskal}. Since then, history has not recorded any thing about self-similarity until the late 19th century when Karl Weierstrass introduced in 1872 a function that being everywhere continuous but nowhere differentiable. The graph of the Weierstrass function became an example of a self-similar curve. The set constructed by Georg Cantor in 1883 is considered as the most essential and influential self similar set since it is a simple and perfect example for theory and applications of this field. Space-filling curves are substantial epitomai of continuous self-similar curves which were described by Giuseppe Peano and David Hilbert in 1890-91. Other examples of self-similar sets are Koch curve discovered by Helge von Koch in 1904 and Sierpinski gasket and carpet which are introduced by Waclaw Sierpinski in 1916. Julia sets gained significance in being generated using the dynamics of iterative function. They are discovered by Gaston Julia and Pierre Fatou in 1917-19, where they studied independently the iteration of rational functions in the complex plane. The term “\textit{fractal}” was coined by Benoit Mandelbrot in 1975 \cite{Mandelbrot0} to describe certain geometrical structures that exhibit self-similarity. Since then, this word has been employed to denote all the above mentioned sets, and the field became known as \textit{fractal geometry}. Consequently, the fractal concept is axiomatically linked with the notion of self-similarity which is considered to be one of the acceptable definitions of fractals. That is, a fractal can be defined as a set that display self-similarity at all scales. However, Mandelbrot define a fractal as a set whose Hausdorff dimension strictly larger than its topological dimension \cite{Mandelbrot1}.	To sum up, self-similarity and fractional dimension are the most two important features of fractals. The connection between them is that self-similarity is the easiest way to construct a set that has fractional dimension \cite{Crownover}.
		
	Chaos, in general,  can be defined as aperiodic long-term behavior in a deterministic system that exhibits sensitive dependence on initial conditions \cite{Kellert}. The first recognition of chaos phenomenon was indicated in the work of Henri Poincar\`{e} in 1890 when he studied the problem of the stability of the solar system. In the 1950s, Edward Lorenz discovered \textit{sensitivity to initial conditions} in a weather forecasting model. This property is considered as a main ingredient of chaos. There are different types and definitions of chaos. Devaney \cite{Devaney} and Li-Yorke \cite{Yorke} chaos are the most frequently used types, which are characterized by transitivity, sensitivity, frequent separation and proximality. Another common type occurs through period-doubling cascade which is a sort of route to chaos through local bifurcations \cite{Feigenbaum,Scholl,Sander}. In the papers \cite{AkhmetUnpredictable,AkhmetPoincare}, Poincaré chaos was introduced through the unpredictable point concepts. Further, it was developed to unpredictable functions and sequences.
	
	Several researches pointed out that a close relationship between chaos and fractal geometry can be observed. It can be seen, for instant, in the dynamics of Fatou-Julia iteration used to construct Julia and Mandelbrot sets where two neighbor points in the domain which are close to the boundary may have completely different behavior. That is, we can say about sensitivity in fractal structures. Chaos tells us about the state of irregularity and divergence of trajectories which depend in the nature of the dynamics, whereas the fractal concept can be used to study complex geometric structures. Therefore, the interlink between chaos and fractals is more clear when fractal dimension is used to measure the extent to which a trajectory fills its phase space. In other words, fractal dimension of the orbit in phase space implies the existence of a strange attractor \cite{Moon}. The fundamental work on the chaotic nature of fractals has been done only for specific types categorized under the totally disconnected fractals \cite{BarnsleyB,Devaney,Chen}. In that work, the topological conjugacy concept was utilized to prove that these fractal sets are invariant for certain chaotic maps. Except for that, relatively few studies have been carried out on chaotic dynamical systems for fractals, and perhaps the most relevant one is what have been done on the Sierpinski carpet in \cite{Ercai}. In that research, the author shows that the dynamical system associated with a shift transformation defined on the Sierpinski carpet set is chaotic in the sense of topological mixing. Our results have a special importance that we do not utilize topological conjugacy to prove the presence of chaos in fractal sets. Moreover, we consider different types of chaos, namely, Devaney, Li-Yorke and Poincaré, and the results cover several kinds of self-similar fractals such as Cantor sets, Sierpinski fractals and Koch curve. Chaos in fractals, particularly Sierpinski carpet and Koch curve, has not been considered in the literature before.  Furthermore, our approach is applicable for sets in multi-dimensional spaces. A good example is the logistic map. It was shown that a chaos equivalent to Li-Yorke type can be extended to higher-dimensional discrete systems \cite{Diamond,Marotto,Dohtani}. This requires employing special theorems like Marotto Theorem \cite{Marotto}. Applying chaotic abstract similarity developed in our paper, we have shown that Devaney, Li-Yorke and Poincaré chaos can take place in the dynamics of $ n $ connected perturbed logistic maps. Examples with numerical simulations are provided in Section \ref{DynConstructOfASSS}.
	
	In this paper we concern with self-similarity. This concept is reflected in many problems that arise in various fields such as wavelets, fractals, and graph systems \cite{Jorgensen}. Interesting definitions of self-similarity and related problems of dimension and measure are discussed in papers \cite{Moran,Hutchinson,Hata,Falconer,Edgar,Spear,Bandt,Falconer1,Lau,Ngai}. The manuscripts consider sets in Euclidean space $ \mathbb{R}^n $ and define self-similar set as a union of its images under similarity transfunctions \cite{Falconer1}. Our research develops self-similarity for metric spaces and this why we call it \textit{abstract self-similarity}. The development does not rely on any special functions, and the similarity map used in our paper is applied for chaos formation. The map is different than any similarity map essentially considered in literature before. In the present paper, we are not concerned in analysis of dimension and measure but mostly rely on the distance, since our main goal is to discuss chaos problem for fractals. Nevertheless, we suppose that our suggestion may be useful for the next extension of the results obtained in \cite{Hutchinson,Hata,Falconer,Edgar,Spear,Bandt,Falconer1,Lau,Ngai} for the abstract self-similarity case. First of all those which consider dimension and measure and corresponding problems of fractals.
	
	We define a self-similar set as a collection of points in a metric space, where it can be considered as a union of infinite shrinking sets with notation that allows to introduce dynamics in the set and then prove chaos which is the most important result of the present research. For the purpose, a specific map over the invariant self-similar set is defined. The map acting as the identity if the whole set taken as an argument for the map. This feature is equivalent to the property of self-similarity in the ordinary fractals, and this is the reason for calling this map the similarity map. We expect that the concept of abstract self-similarity will create new frontiers for chaos and fractals investigations, and we hope it will be a helpful tool in other fields such as harmonic analysis, discrete mathematics, probability, and operator algebras \cite{Jorgensen}.  
	
	Our approach with respect to chaos is characterized by the priority of the domain over the map of chaos. More precisely, the usual construction of chaos starts with description of a map with certain properties like unimodality, hyperbolicity, period three and topological conjugacy to a standard chaotic map. Next, chaotic behavior is discovered, and finally, the structure of the chaotic attractor is analyzed. For example, in Li-Yorke chaos, a map is firstly defined such that it has period three property, and then, the domain (scrambled set) for the chaos is characterize by specific properties. The same can be said for the chaos of unimodal maps, when the domain of chaos is a Cantor set. In the Smale horseshoe case, we can conclude that the map and the domain are simultaneously determined so that the construction of the domain started with an initial set and it is developed step by step using a particular map. Therefore, the structure of the domain and the nature of the map are mutually dependent on each other. For the chaos in symbolic dynamics, the domain is primarily described as infinite sequences of symbols then the map is introduced as a shift on the space. However, the map still has priority since the properties of the sequences are described with respect to the map. In the proposed approach, we first construct a domain in a metric space with specific conditions to be a suitable venue for manifestations of chaos. Thereafter, the similarity map is built on the basis of the invariance and self-similarity properties of the domain to define an abstract motion. This is why the map is appropriate for abstract self-similar set as well as for any fractal constructed through self-similarity, and thence proving chaos for these classes of fractals becomes possible. Moreover, we drop the continuity requirement for the motion since the chaotic map need not be continuous \cite{Li,Addabbo}. In paper \cite{Li}, for instance, the authors ignore the continuity of some chaotic maps during the discussion of chaos conditions on the product of semi-flows. We regard the continuity of a chaotic map as an important property only from the analytical side, that is to say, it is very useful for handling the map to prove presence of chaos \cite{Wiggins,Chen}, however, it is not rigorously correct to consider it as an intrinsic property for chaos. Despite the discontinuity of the similarity map, opposite to our desire, one can recognize that presence of sensitivity and the irregular behavior of simulations make the discussion precious for theoretical investigations as well as for future applications.  
	
	\section{Abstract Self-Similarity} \label{DefnASSS}

	Let us consider the metric space $ (\mathcal{F}, d) $, where $ \mathcal{F} $ is a compact set and $ d $ is a metric. We assume that $ \mathcal{F} $ is divided into $ m $ disjoint nonempty subsets, $ \mathcal{F}_i, \; i=1, 2, ..., m $, such that $ \mathcal{F} = \cup_{i=1}^{m} \mathcal{F}_i $. In their own turn, the sets $ \mathcal{F}_i \; i=1, 2, ..., m $, are divided into $ m $ disjoint nonempty subsets $ \mathcal{F}_{ij}, \; j=1, 2, ..., m $, such that $ \mathcal{F}_i = \cup_{j=1}^{m} \mathcal{F}_{ij} $. That is, in general we have $ \mathcal{F}_{i_1 i_2 ... i_n} = \cup_{j=1}^{m} \mathcal{F}_{i_1 i_2 ... i_n j} $, for each natural number $ n $, where all sets $ \mathcal{F}_{i_1 i_2 ... i_n j}, \; j=1, 2, ..., m $, are nonempty and disjoint. We assume that for the sets $ \mathcal{F}_{i_1 i_2 ... i_n} $, the \textit{diameter condition} is valid. That is
	\begin{equation} \label{Diamprop}
	\max_{i_k=1,2, ..., m} \mathrm{diam}(\mathcal{F}_{i_1 i_2 ... i_n}) \to 0 \;\; \text{as} \;\; n \to \infty,
	\end{equation}
	where $ \mathrm{diam}(A) = \sup \{ d(\textbf{x}, \textbf{y}) : \textbf{x}, \textbf{y} \in A \} $, for a set $ A $ in $ \mathcal{F} $. 
	
	Let us construct a sequence, $ p_n $, of points in $ \mathcal{F} $ such that $ p_0 \in \mathcal{F} $, $ p_1 \in \mathcal{F}_{i_1} $, $ p_2 \in \mathcal{F}_{i_1 i_2} $, ... , $ p_n \in \mathcal{F}_{i_1 i_2 ... i_n}, \; n=1, 2, ...\, $. It is clear that,
	\[ \mathcal{F} \supset \mathcal{F}_{i_1} \supset \mathcal{F}_{i_1 i_2} \supset ... \supset \mathcal{F}_{i_1 i_2 ... i_n} \supset \mathcal{F}_{i_1 i_2 ... i_n i_{n+1}} ... , \; i_k=1, 2, ... , m, \; k=1, 2, ... \, . \]
	That is, the sets form a nested sequence. Therefore, due to the compactness of $ \mathcal{F} $ and the diameter condition, there exists a unique limit point for the sequence $ p_n $. Denote the point as $ \mathcal{F}_{i_1 i_2 ... i_n ... } \in \mathcal{F} $, accordingly to the indexes of the nested subsets. Conversely, it is easy to verify that each point $ p \in \mathcal{F} $ admits a corresponding $ p_n $ and it can be written as $ p = \mathcal{F}_{i_1 i_2 ... i_n ...} $, and this representation is a unique one due to the diameter condition. Finally, we have that
	\begin{equation} \label{AbstFracSet}
	\mathcal{F} =  \big\{\mathcal{F}_{i_1 i_2 ... i_n ... } : i_k=1,2, ..., m, \; k=1, 2, ... \big\}, 
	\end{equation}
	and
	\begin{equation} \label{AbstFracSubSet}
	\mathcal{F}_{i_1 i_2 ... i_n} = \bigcup_{j_k=1,2, ..., m } \mathcal{F}_{i_1 i_2 ... i_n j_1 j_2 ... },
	\end{equation}
	for fixed indexes $ i_1, i_2, ..., i_n $.
	
	The set $ \mathcal{F} $ satisfies (\ref{AbstFracSet}) and (\ref{AbstFracSubSet}) is said to be the \textit{abstract self-similar set} as well as the triple $ (\mathcal{F}, d, \varphi) $ the \textit{self-similar space}.

	Let us introduce the map $ \varphi : \mathcal{F} \to \mathcal{F} $ such that
	\begin{equation} \label{MapDefn}
	\varphi (\mathcal{F}_{i_1 i_2 ... i_n ... }) = \mathcal{F}_{i_2 i_3 ... i_n ... }.
	\end{equation}
	Considering iterations of the map, one can verify that
	\begin{equation} \label{MapSubset}
	\varphi^n(F_{i_1 i_2 ... i_n}) = \mathcal{F},
	\end{equation}
	for arbitrary natural number $ n $ and $ i_k=1,2, ..., m, \; k=1, 2, ... \, $. The relations (\ref{MapDefn}) and (\ref{MapSubset}) give us a reason to call $ \varphi $ a \textit{similarity map} and the number $ n $ the \textit{order of similarity}.
	
	In the next example of our paper and in the future studies, it is important to find the structure of abstract self-similar space for a given mathematical object.
	
	\begin{example} \label{Exm1}
		Let us consider the space of symbolic strings of $ 0 $ and $ 1 $ \cite{Wiggins,Devaney}, which is defined  by
		\[ \varSigma=\{ s_1 s_2 s_3 ... : s_k=0 \, \text{or} \, 1 \}. \]
		The distance in $ \varSigma $ is defined by
		\begin{equation} \label{DistMetric}
		d(s, t) = \sum_{k=1}^{\infty} \frac{|s_k-t_k|}{2^{k-1}}, 
		\end{equation}
		where $ s= s_1 s_2 ... $ and $ t= t_1 t_2 ... $ be two elements in $ \varSigma $.
		 
		Considering the pattern of the self-similar set, we denote the elements of the set by $ \varSigma_{s_1 s_2 ...} = s_1 s_2 ... $, and describe the $ n $th order subsets of strings in $ \varSigma $ by
		\[ \varSigma_{s_1 s_2 ... \, s_n} = \{s_1 s_2 ... s_n s_{n+1} s_{n+2} ... : s_k=0 \, \text{or} \, 1 \}, \]
		where $ s_1, s_2, ..., s_n $ are fixed symbols.
		One can show that $ d(s, t) \leq \frac{1}{2^{n-1}} $ for any two elements $ s, t \in \varSigma_{s_1 s_2 ... s_n} $. Moreover, $ d(s_1 s_2 ... s_n 0 0 0 ... \, , s_1 s_2 ... s_n 1 1 1 ...) = \frac{1}{2^{n-1}} $. Therefore, $ \mathrm{diam}(\varSigma_{s_1 s_2 ... s_n}) = \frac{1}{2^{n-1}} $. Consequently,
		\[ \lim_{ n \to \infty} \mathrm{diam}(\varSigma_{s_1 s_2 ... s_n}) = \lim_{ n \to \infty}  \frac{1}{2^{n-1}} =0, \]
		and the diameter condition holds.
		
		The similarity map for the space is the Bernoulli shift, $ \sigma (s_1 s_2 s_3 ... ) = s_2 s_3 s_4 ... $. That is,
		\[ \varphi (\varSigma_{s_1 s_2 s_3 ... }) = \sigma (s_1 s_2 s_3 ... ). \]
		
		On the basis of the above discussion, one can conclude that the triple $ (\varSigma, d, \varphi) $ is a self-similar space. This is a purely illustrative example since it makes us perceive how self-similarity can be defined for abstract objects which are not necessarily geometrical ones. The space of symbolic strings on two symbols has been considered, since it is the most basic example that frequently used to describe the dynamics on symbolic spaces. However, more generally, the space on $ m $ symbols can also be considered.
		  	
	\end{example}

	\section{Similarity and Chaos} \label{AFChaos}
	
	To prove chaos for the self-similar space, we assumed in this section the \textit{separation condition}. Define the distance between two nonempty bounded sets $ A $ and $ B $ in $ \mathcal{F} $ by $ d(A, B)= \inf \{ d(\textbf{x}, \textbf{y}) : \textbf{x}\in A, \, \textbf{y} \in B \} $. The set $ \mathcal{F} $ satisfies the separation condition of degree $ n $ if there exist a positive number $ \varepsilon_0 $ and a natural number $ n $  such that for arbitrary $ i_1 i_2 ... i_n $ one can find $ j_1 j_2 ... j_n $ so that
	\begin{equation} \label{C2}
	d \big( \mathcal{F}_{i_1 i_2 ... i_n} \, , \, \mathcal{F}_{j_1 j_2 ... j_n} \big) \geq \varepsilon_0.
	\end{equation}
	We call $ \varepsilon_0 $ the \textit{separation constant}.	
	
	In the following theorem, we prove that the similarity map $ \varphi $ possesses the three ingredients of Devaney chaos, namely density of periodic points, transitivity and sensitivity. A point $ \mathcal{F}_{i_1 i_2 i_3 ...} \in \mathcal{F} $ is periodic with period $ n $ if its index consists of endless repetitions of a block of $ n $ terms.

	\begin{theorem} \label{Thm1}
	If the separation condition holds, then the similarity map is chaotic in the sense of Devaney.
	\end{theorem}

	\begin{proof}
	Fix a member $ \mathcal{F}_{i_1 i_2 ... i_n ... } $ of $ \mathcal{F} $ and a positive number $ \varepsilon $. Find a natural number $ k $ such that $ \mathrm{diam}(\mathcal{F}_{i_1 i_2 ... i_k}) < \varepsilon $ and choose a $ k $-periodic element $ \mathcal{F}_{i_1 i_2 ... i_k i_1 i_2 ... i_k ...} $ of $ \mathcal{F}_{i_1 i_2 ... i_k} $. It is clear that the periodic point is an $ \varepsilon $-approximation for the considered member. The density of periodic points is thus proved.
	
	Next, utilizing the diameter condition, the transitivity will be proved if we show the existence of an element $ \mathcal{F}_{i_1 i_2 ... i_n ...} $ of $ \mathcal{F} $ such that for any subset $ \mathcal{F}_{i_1 i_2 ... i_k} $ there exists a sufficiently large integer $ p $ so that $ \varphi^p(\mathcal{F}_{i_1 i_2 ... i_n ...}) \in \mathcal{F}_{i_1 i_2 ... i_k} $. This is true since we can construct the sequence $ i_1 i_2 ... i_n ... $ such that it contains all sequences of the type $ i_1 i_2 ... i_k $ as blocks.
		
	For sensitivity, fix a point $ \mathcal{F}_{i_1 i_2 ... } \in \mathcal{F} $ and an arbitrary positive number $ \varepsilon $. Due to the diameter condition, there exist an integer $ k $ and element $ \mathcal{F}_{i_1 i_2 ... i_k j_{k+1} j_{k+2} ...} \neq \mathcal{F}_{i_1 i_2 ... i_k i_{k+1} i_{k+2} ...} $ such that $ d(\mathcal{F}_{i_1 i_2 ... i_k i_{k+1} ...}, \mathcal{F}_{i_1 i_2 ... i_k j_{k+1} j_{k+2} ...}) < \varepsilon $. We precise $ j_{k+1}, j_{k+2}, ... $ such that \[ d(\mathcal{F}_{i_{k+1} i_{k+2} ... i_{k+n}}, \mathcal{F}_{j_{k+1} j_{k+2} ... j_{k+n}} \allowbreak ) > \varepsilon_0, \]
	by the separation condition. This proves the sensitivity.	
	\end{proof}

	For Poincar\`{e} chaos, Poisson stable motion is utilized to distinguish the chaotic behavior instead of the periodic motions in Devaney and Li-Yorke types. Existence of infinitely many unpredictable Poisson stable trajectories that lie in a compact set meet all requirements of chaos. Based on this, chaos can be appeared in the dynamics on the quasi-minimal set which is the closure of a Poisson stable trajectory. Therefore, the Poincar\`{e} chaos is referred to as the dynamics on the quasi-minimal set of trajectory initiated from unpredictable point. For more details we refer the reader to \cite{AkhmetUnpredictable,AkhmetPoincare}.
	
	Next theorem shows that the Poincar\`{e} chaos is valid for the similarity dynamics.

	\begin{theorem} \label{Thm2}
	If the separation condition is valid, Then the similarity map possesses Poincar\`{e} chaos .
	\end{theorem}

	The proof of the last theorem is based on the verification of Lemma 3.1 in \cite{AkhmetPoincare} adopted to the similarity map.
	
	In addition to the Devaney and Poincar\`{e} chaos, it can be shown that the Li-Yorke chaos also takes place in the dynamics of the map $ \varphi $. The proof of the following theorem is similar to that of Theorem 6.35 in \cite{Chen} for the shift map defined on the space of symbolic sequences.
	 
	\begin{theorem} \label{Thm3}
  	The similarity map is Li–Yorke chaotic if the separation condition holds.
	\end{theorem}
	
	\begin{example}
	We have shown that the space of symbolic strings is a self-similar set in Example \ref{Exm1}. One can see that $ \varSigma = \varSigma_0 \cup \varSigma_1 $, where $ \varSigma_0= \{ 0 s_2 s_3 ... \} $ and $ \varSigma_1= \{ 1 s_2 s_3 ... \} $, hence,
	
	\begin{equation*} \label{VDP System}
	\begin{split}
		d(\varSigma_0, \varSigma_1)&= \inf \{ d(s, t) : s \in \varSigma_0, \, t \in \varSigma_1 \} \\
	&= d(0 0 0 ... \, , 1 0 0 0 ...) \\
	&= d(0 1 1 1 ... \, , 1 1 1 ...) = 1.
	\end{split}	
	\end{equation*}
	Therefore, the separation condition of degree $ 1 $ holds with the separation constant $ \varepsilon_0 $ equal to unity.
	
	According to the results of this section, the Bernoulli shift is chaotic in the sense of Poincar\`{e}, Li-Yorke and Devaney. That is, we confirm one more time the presence of chaos which have been proven for the dynamics in \cite{Wiggins,Devaney,AkhmetReplication,AkhmetPoincare}.
	\end{example}

	\section{Chaos in Fractals}
	
	As implementations of abstract self-similarity, we consider several examples of fractals namely Sierpinski carpet, Sierpinski gasket, Koch curve and Cantor set. This consists of two main tasks. The first one is to indicate abstract self-similarity for the fractals, and the second one is to ascertain chaos according to the results of the last section.
		
	\subsection{Chaos for Sierpinski carpet} \label{SCexm}	

	Let $ S $ be the Sierpinski carpet constructed in a unit square. In what follows, we are going to find the structure of the abstract self-similar space for the Sierpinski carpet. We shall denote the abstract set by the italic $ S $. Let us start by dividing the carpet into eight subsets and denote them as $ S_1, S_2, ... S_8 $ (see Fig. \ref{SCChaos1} (a)). The subsets will be determined such that any couple of adjacent subsets have common horizontal or vertical boundary line. For this reason, we shall use the \textit{boundary agreement} such that: $ (i) $ The points of the common boundary of two horizontally adjacent subsets belong to the left one. $ (ii) $ The points of the common boundary of two vertically adjacent subsets belong to the lower one. Figure \ref{SCChaos1} (b) illustrates the boundary agreement, $ (i) $ and $ (ii) $, for the subsets, $ S_5 $, $ S_7 $ and $ S_8 $. For clarification the boundaries are shown by black lines and we see that the common boundary points of $ S_5 $ and $ S_8 $ belong to $ S_5 $ not to $ S_8 $ and the common boundary points of $ S_7 $ and $ S_8 $ belong to $ S_7 $ not to $ S_8 $. In the second step, each subset $ S_i, \, i=1,2, ..., 8 $ is again subdivided into eight smaller subsets denoted as $ S_{ij}, \, j=1,2, ..., 8 $.
	\begin{figure}[H]
	\centering
	\subfigure[]{\includegraphics[width=0.40\linewidth]{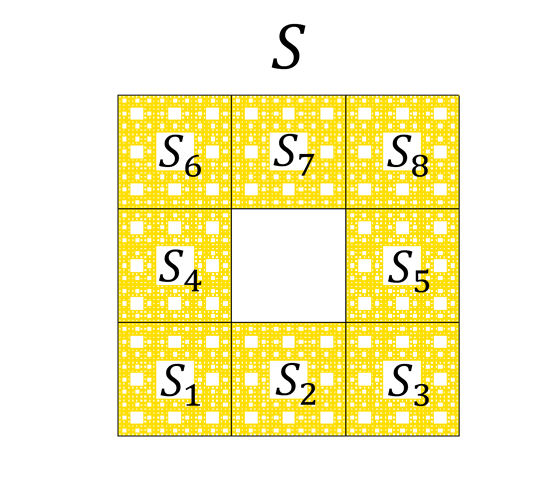}}\hspace{1cm}
	\subfigure[]{\includegraphics[width=0.30\linewidth]{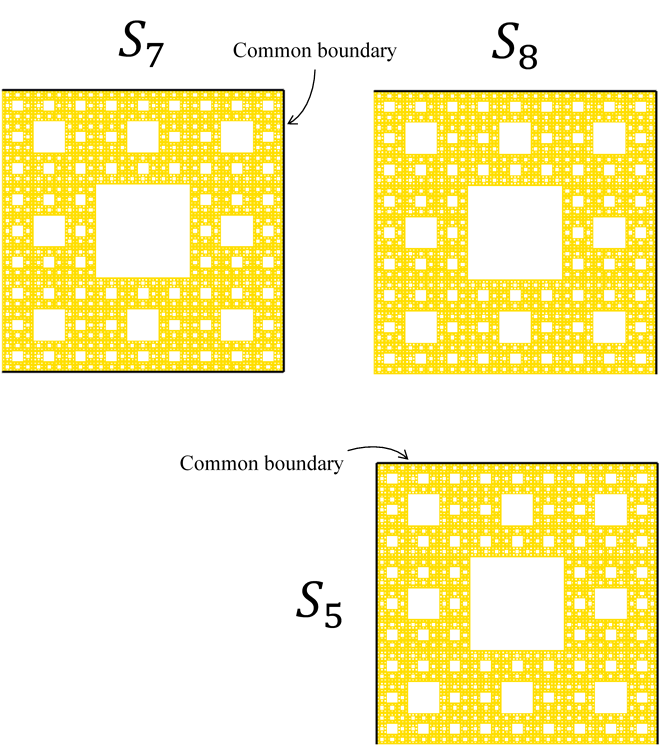}}
	\caption{(a) The first step of abstract self-similar set construction. (b) The illustration of the boundary agreement.}
	\label{SCChaos1}	
	\end{figure}
	Continuing in the same manner, the subsets of higher order can inductively be determined such that at each $ n^{th} $ step we have $ 8^n $ subsets notated as $ S_{i_1 i_2 ... i_n}, \, i_k=1,2, ..., 8 $. Figure \ref{SCChaos2} (a) and (b) show, for example, the subsets of $ S_1 $ and subsets of $ S_{11} $ respectively.
	
	\begin{figure}[H]
	\centering
	\subfigure[]{\includegraphics[width=0.36\linewidth]{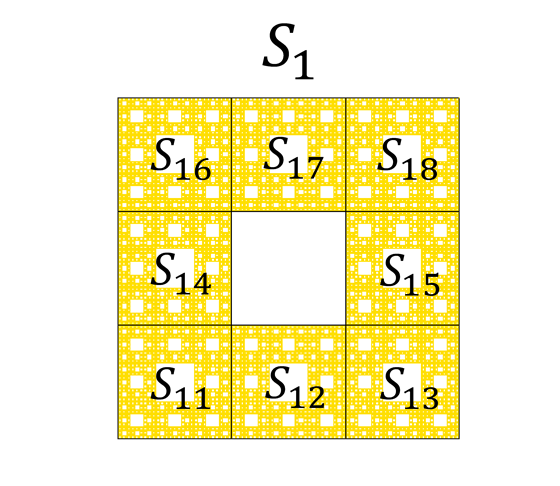}}\hspace{1cm}
	\subfigure[]{\includegraphics[width=0.32\linewidth]{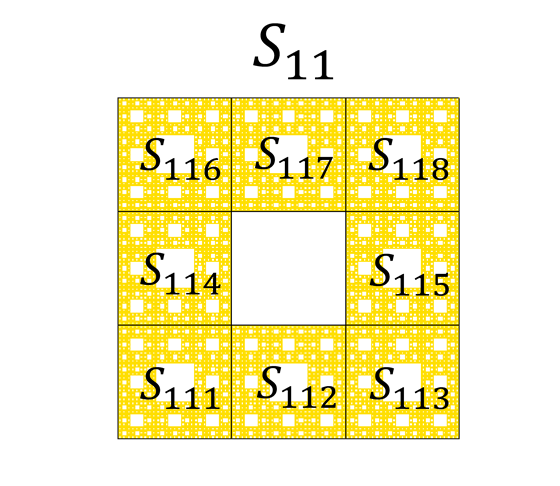}}
	\caption{Examples of the $ 2^{nd} $ and the $ 3^{rd} $ order subsets of the Sierpinski carpet}
	\label{SCChaos2}	
	\end{figure}
	
	To determine the distance between the points of $ S $, we will apply the corresponding Euclidean distance for the set $ S $ such that if $ \textbf{x}=(x_1, x_2) $ and $ \textbf{y}=(y_1, y_2) $ are two points in $ S $, then, $ d(\textbf{x}, \textbf{y})=\sqrt{(x_1-y_1)^2+(x_2-y_2)^2} $. The diameter of a subset at an $ n^{th} $ step is equal to, $ \mathrm{diam}(S_{i_1 i_2 i_3 ... i_n})=\frac{\sqrt{2}}{3^n} $, and therefore, it diminishes to zero as  $ n $ tends to infinity, and the diameter condition holds. It is easy to check that each point in $ S $ has a unique presentation $ S_{i_1 i_2 ... i_n ... } $. Hence, the set $ S $ can be written as
	\[ S =  \big\{S_{i_1 i_2 ... i_n ... } : i_k=1,2, ... 8, \; k \in \mathbb{N} \big\}. \]
	
	The separation condition of degree $ 1 $ is satisfied and the separation constant is
	\[ \varepsilon_0 = \big\{ \min \{ d(S_i, S_j) \} : S_i, S_j \text{ are disjoint}, \, i,j=1,2, ..., 8 \big\} = \frac{1}{3}. \]

	Let us now define the similarity map by
	\[ \varphi(S_{i_1 i_2 i_3 ... })=S_{i_2 i_3 ... }. \]
	Thus, we have shown that the triple $ (S, d, \varphi) $ is a self similar space with the separation condition.  In view of Theorems \ref{Thm1}, \ref{Thm2} and \ref{Thm3}, the similarity map $ S $ is chaotic in the sense of Poincar\'{e}, Li-Yorke and Devaney.
	
	\subsection{A chaotic trajectory in the Sierpinski carpet}
	
	In this section, we provide a geometric realization of the similarity map on the Sierpinski carpet and see how the map can be useful for visualizing the trajectories of the points of a self-similar set and indexing its subsets. A chaotic trajectory is seen as expected in the last section. In the paper \cite{Akhmet}, we adopt the idea of Fatou-Julia iteration (also called Escape Time Algorithm (ETA)\cite{BarnsleyB}) and develop a scheme for constructing the Sierpinski carpet. The scheme is based on the iterations of  the modified planar tent map
	
	\begin{equation*} \label{ModTentMap}
	\begin{split}	
	&T(x) = \left\{ \begin{array}{ll}\vspace{2mm}
	3 \, [x (\text{mod} 1)] & \quad x \leq \frac{1}{2} \; \text{or} \; x > 1, \\ 
	3 (1-x) & \quad \frac{1}{2} < x \leq 1,
	\end{array}	\right.\\
	&T(y) = \left\{ \begin{array}{ll}\vspace{2mm}
	3 \, [y (\text{mod} 1)] & \quad y \leq \frac{1}{2} \; \text{or} \; y > 1, \\ 
	3 (1-y) & \quad \frac{1}{2} < y \leq 1.
	\end{array}	\right.
	\end{split}			
	\end{equation*}
	Depend on this, one can construct a map $ \bar{T}=(\bar{T}_1, \bar{T}_2): S \to S $ such that the set $ S $ is invariant,
	
	\begin{equation} \label{InvModTentMap}
	\begin{split}	
	&\bar{T}_1(x) = \left\{ \begin{array}{ll} \vspace{2mm}
	3 x & \quad 0 \leq x \leq \frac{1}{3}, \\  \vspace{2mm}
	3 x-1 & \quad \frac{1}{3} < x \leq \frac{1}{2}, \\ \vspace{2mm}
	2-3(1-x) & \quad \frac{1}{2} < x < \frac{2}{3}, \\ 
	3 (1-x) & \quad \frac{2}{3} \leq x \leq 1,
	\end{array}	\right.\\
	&\bar{T}_2(y) = \left\{ \begin{array}{ll} \vspace{2mm}
	3 y & \quad 0 \leq y \leq \frac{1}{3}, \\  \vspace{2mm}
	3 y-1 & \quad \frac{1}{3} < y \leq \frac{1}{2}, \\ \vspace{2mm}
	2-3(1-y) & \quad \frac{1}{2} < y < \frac{2}{3}, \\ 
	3 (1-y) & \quad \frac{2}{3} \leq y \leq 1.
	\end{array}	\right.
	\end{split}			
	\end{equation}
	
	This map is equivalent to the similarity map $ \varphi $ defined above, therefore, the trajectory of a point $ \textbf{x} \in S $ can be visualized using the map (\ref{InvModTentMap}). Figure \ref{ShiftTraj} shows an example of the trajectory for the center point $ \textbf{x} $ of the subset
	\[ S_{27731137313277182431515822461784764852656358462545627125423317216244}. \]
	The points of the trajectory are considered as the centers of the subsets which are determined by $ \bar{T}^k (\textbf{x}), \; k=0, 1, 2, ... \, 68 $. The idea of indexing of the subsets is illustrated in Example \ref{Exm3}.
	
	\begin{figure}[H]
	\centering
	\includegraphics[width=0.50\linewidth]{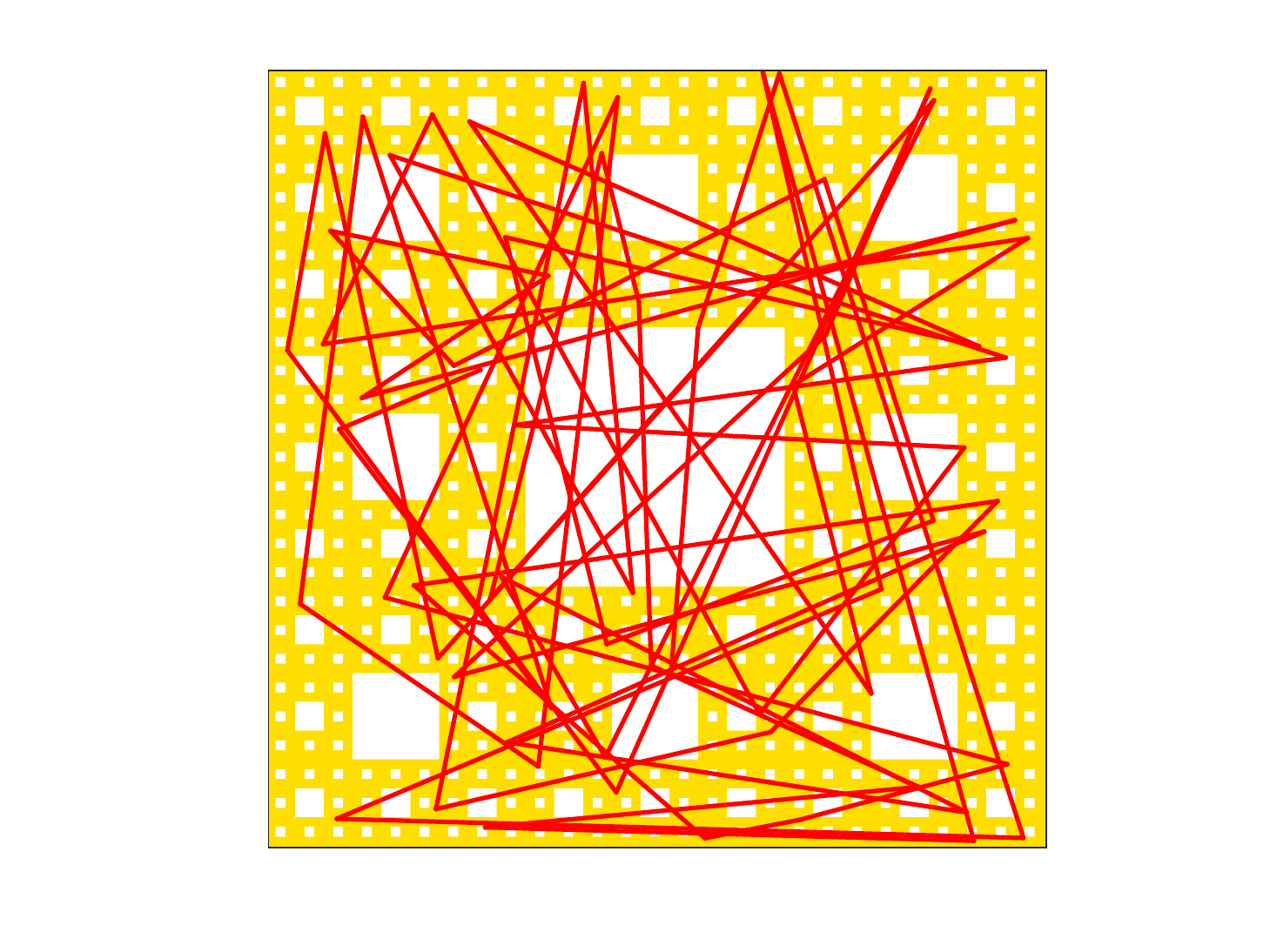}
	\caption{The trajectory of a point under the similarity map}
	\label{ShiftTraj}
	\end{figure}

	\subsection{Sierpinski gasket as chaos domain}
	
	To construct an abstract self-similar set on the basis of the Sierpinski gasket, let us consider the Sierpinski gasket generated in a unit equilateral triangle. We firstly divided the gasket into three smaller parts to be the first order subsets and denoted them by $ G_1 $, $ G_2 $ and $ G_3 $ as shown in Fig. \ref{SGChaos} (a). By glancing at the figure, one can see that every two subsets share only a single point as a common boundary. For this reason we consider the following boundary agreement: The common boundary point of every couple of adjacent subsets belongs either to the left one or to the lower one. Applying the agreement, the subsets $ G_1 $, $ G_2 $ and $ G_3 $, become disjoint subsets of the desired abstract self-similar set $ G $ such that $ G = \cup_{i=1}^{3} G_i $. Secondly, each subset, $ G_i, \; i=1, 2, 3 $, is again subdivided into three subsets, and we notate them as $ G_{ij}, \; i,j=1, 2, 3 $, (see Fig. \ref{SGChaos} (b)). Taking into account the boundary agreement, we repeat the same procedure such that at each $ n^{th} $ step, we denote the resultant subsets by $ G_{i_1 i_2 ... i_n}, \; i_k=1, 2, 3 $.
	
	The subsets of the Sierpinski gasket described above have an inverse relationship with the construction-step variable, $ n $, $ \mathrm{diam}(S_{i_1 i_2 i_3 ... i_n})=\frac{1}{2^n} $, from which one can verify the validity of the diameter condition.

	Following the arguments of the abstract self-similarity, one can deduce that a point in $ G $ can be uniquely represented by $ G_{i_1 i_2 ... i_n ... } $, and the abstract self-similar set $ G $ can be expressed as
	\[ G =  \big\{G_{i_1 i_2 ... i_n ... } : i_k=1,2,3, \; k \in \mathbb{N} \big\}. \]
	
	\begin{figure}[H]
	\centering
	\subfigure[]{\includegraphics[width = 2.5in]{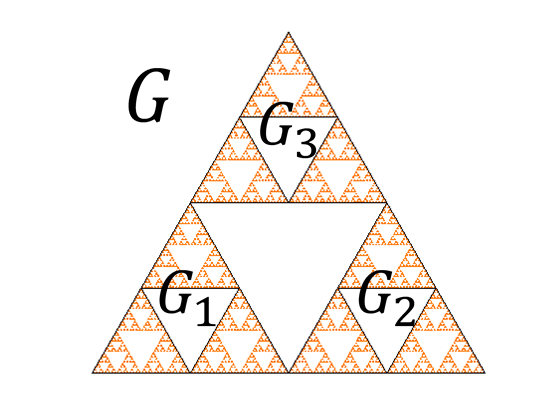}\label{SFG}} 
	\subfigure[]{\includegraphics[width = 2.5in]{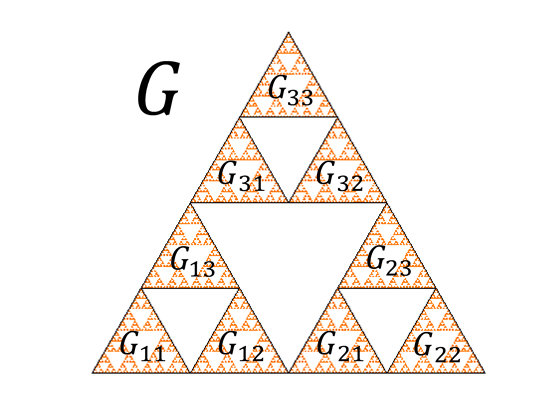}\label{SFC}}
	\caption{The construction of abstract self-similar set corresponding to the Sierpinski gasket}
	\label{SGChaos}   				
	\end{figure}

	Contrary to the Sierpinski carpet case, the separation constant $ \varepsilon_0 $ for the Sierpinski gasket cannot be evaluated through the first order subsets $ G_i, \; i=1, 2, 3 $. This is why we considered the minimum of the distance between any two disjoint subsets of the second order, that is
	\[ \varepsilon_0 = \{ \min \{ d(G_{i_1 i_2}, G_{j_1 j_2}) \} : G_{i_1 i_2}, G_{j_1 j_2} \text{ are disjoint}, \; i_n=1, 2, 3 \} = \frac{\sqrt{3}}{8}, \]
	where $ d $ is the usual Euclidean distance. Thus, one can see that separation condition is valid.
	
	The similarity map acting on the Sierpinski gasket, $ G $, can be defined by $ \varphi(G_{i_1 i_2 i_3 ...}) = G_{i_2 i_3 ...} $. Consequently, the triple $ (G, d, \varphi) $ is a self-similar space and $ \varphi $ is chaotic in the sense of Poincar\'{e}, Li-Yorke and Devaney.
	
	The same idea can be extended to the fractals associated with Pascal's triangles. It is well known that Pascal's triangle in mod 2 creates the classical Sierpinski gasket. Different fractals associated with Pascal's triangles in different moduli can be considered as abstract similar sets and it can also be proved that the similarity map defined on these sets possesses chaos.
	
	\subsection{Koch curve and chaos}
	
	Let us consider the Koch curve, $ K $, constructed from an initial unit line segment. To identify an abstract self-similar set corresponding to the Koch curve, we start by dividing $ K $ into four equal parts (subsets) and denoting them as $ K_1, K_2, K_3 $ and $ K_4 $ as shown in Fig. \ref{KCConst} (a). Since the Koch curve is a connected set, each two adjacent subsets share a single point as a common boundary. Let us denote the end points of each subset $ K_i, \; i=2,3 $ by $ a_i $ and $ a_{i+1} $. Figure \ref{KCConst} (a) illustrates these points and for clarification the points are shown by black thick dots. It is seen in the figure that $ K_1 $ and $ K_2 $ share the point $ a_2 $, $ K_2 $ and $ K_3 $ share the point $ a_3 $ and $ K_3 $ and $ K_4 $ share the point $ a_4 $. In the second step, each subset $ K_i, \; i=1,2,3,4 $ is again subdivided into four subsets $ K_{ij}, \; j=1,2,3,4 $. Figure \ref{KCConst} (b) and (c) illustrate the second step for the subsets $ K_1 $ and $ K_2 $ respectively. Again here we see that each two adjacent subsets share a single boundary point. We continue in the same way such that at each step, $ n $, every set $ K_{i_1 i_2 ... i_{n-1}}, \; i_k=1, 2, 3, 4 $ is divided into four subsets $ K_{i_1 i_2 ... i_{n-1} i_n}, \; i_n=1, 2, 3, 4 $ and denote the end points of each $ K_{i_1 i_2 ... i_n}, \; i_n=2,3 $, by $ a_{i_1 i_2 ... i_n} $ and $ a_{i_1 i_2 ... i_n+1} $.
	
	As in the previous cases, to determine the abstract self-similar set, we need to consider the following boundary agreement: For each adjacent subsets $ K_{i_1 i_2 ... i_{n-1}j} $ and $ K_{i_1 i_2 ... i_{n-1}j+1} $, the common boundary point $ a_{i_1 i_2 ... i_{n-1}j+1} $ belongs to $ K_{i_1 i_2 ... i_{n-1}j+1} $. This condition means that the common boundary point $ a_2 $ shown in Fig. \ref{KCConst} (a), for instance, belongs to $ K_2 $ not to $ K_1 $ and the common boundary point $ a_{23} $ shown in Fig. \ref{KCConst} (c) belongs to $ K_{23} $ not to $ K_{22} $.
	
	By applying the boundary agreement to all subsets at each step, we have fully described the disjoint subsets of the proposed abstract self-similar set for the Koch curve. From the construction of the Koch curve and by using the usual Euclidean distance, one can deduce that the distance between the end points of each subset $ K_{i_1 i_2 ... i_n} $ is $ \frac{1}{3^n} $ which clearly represents the diameter of the subset. Therefore, the diameter condition holds. A point in $ K $ can be represent by $ K_{i_1 i_2 ... i_n ... } $, so that,
	\[ K =  \big\{K_{i_1 i_2 ... i_n ... } : i_k=1,2,3,4, \; k \in \mathbb{N} \big\}. \] 
	
	\begin{figure}[H]
	\centering
	\subfigure[]{\includegraphics[width = 2.6in]{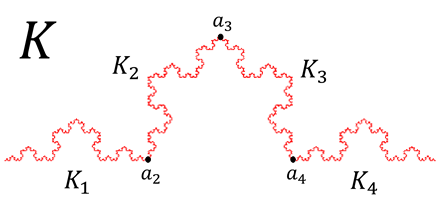}\label{KCChaosa}} \hspace{0cm}
	\subfigure[]{\includegraphics[width = 1.8in]{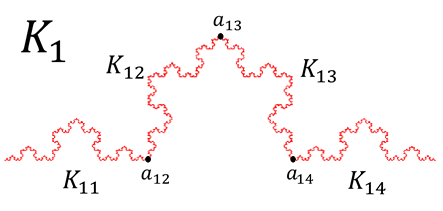}\label{KCChaosa}} \hspace{0cm}
	\subfigure[]{\includegraphics[width = 1.3in]{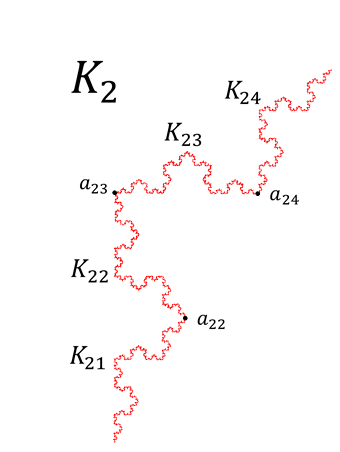}\label{KCChaosb}}
	\caption{The construction of abstract self-similar set corresponding to the Koch curve}
	\label{KCConst}   				
	\end{figure}
	
	The separation condition is also valid with degree 1, since for any $ K_i, \, i=1,2,3,4 $ one can find $ K_j, \, j=1,2,3,4, \, j \neq i $ such that they are separated from each other by a distance of not less than $ \varepsilon_0 $. The separation constant, $ \varepsilon_0 $, can be defined by
	\[ \varepsilon_0 = \min \{ d(K_1, K_3), d(K_1, K_4),d(K_2, K_4) \} = \frac{\sqrt{7}}{9}. \]  

	The similarity map for the abstract fractals of Koch curve is given by $ \varphi(K_{i_1 i_2 i_3 ...}) = K_{i_2 i_3 ...} $, and thus, we have shown that the triple $ (K, d, \varphi) $ defines a chaotic self-similar space.

	\subsection{Chaos for Cantor set}

	A perfect example of chaos in fractals is the Cantor set. As we previously mentioned, the chaoticity in the Cantor set is determined by finding a topological conjugacy with the symbolic dynamics. To show that the Cantor set is not an exception to our approach for chaos, we shall establish an abstract self-similar set corresponding to the Cantor set. Let us consider the middle third Cantor set, $ C $, initiated from a unit line segment. The first step consists of dividing $ C $ into two subsets and denoted them by $ C_1 $ and $ C_2 $ (see Fig. \ref{CSConst} (a)) . In the second step each of $ C_1 $ and $ C_2 $ is subdivided into two subsets as shown in Fig. \ref{CSConst} (b). These subsets are denoted by $ C_{11} $, $ C_{12} $, $ C_{21} $ and $ C_{22} $. In every next step, we repeat the same procedure for each subsets resulting from the preceding step. We denote the resultant subsets at each $ n^{th} $ step by $ C_{i_1 i_2 ... i_n}, \; i_k=1, 2 $.
	
	\begin{figure}[H]
	\centering
	\subfigure[]{\includegraphics[width = 2.6in]{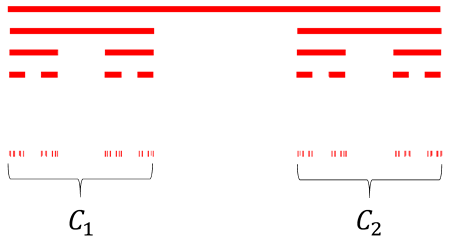}\label{KCChaosa}} \hspace{2cm}
	\subfigure[]{\includegraphics[width = 2.6in]{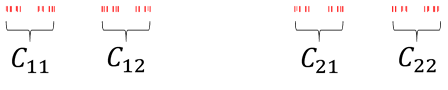}\label{KCChaosa}} 
	\caption{The $ 1^{st} $ and the $ 2^{nd} $ order subsets of the abstract self-similar set for the Cantor set}
	\label{CSConst}   				
	\end{figure}

	In the Cantor set case, we do not need any boundary agreement since all subsets are disjoint. Considering the usual Euclidean distance, the diameter of a subset, $ C_{i_1 i_2 ... i_n} $, is $ \frac{1}{3^n} $. Therefore, the diameter condition holds. The points in $ C $ are represented by $ C_{i_1 i_2 ... i_n ... } $. Hence the abstract self-similar set is defined by
	\[ C =  \big\{C_{i_1 i_2 ... i_n ... } : i_k=1,2, \; k \in \mathbb{N} \big\}. \]
	
	From the construction, the separation condition is clearly valid and the constant $ \varepsilon_0 $ is defined by the distance between the subsets $ C_1 $ and $ C_2 $, so that $ \varepsilon_0 = \frac{1}{3} $.
	
	The similarity map is defined by $ \varphi(C_{i_1 i_2 i_3 ... })=C_{i_2 i_3 ... } $, and the triple $ (C, d, \varphi) $ defines a self-similar space. Theorems \ref{Thm1}, \ref{Thm2} and \ref{Thm3} are also applicable for this case.\\

	In connection with the above examples of chaos, we remark that the Sierpinski carpet and Koch curve indicate that a non-continuous map can have a domain which is a connected set while for continuous maps, the domains of chaos are usually disconnected. Examples of disconnected chaotic domains are the Cantor set for the logistic map \cite{Devaney}, the modified Sierpinski triangle with exceptions in \cite{BarnsleyB}, the set associated with Smale's horseshoe map \cite{Zeraoulia}, and the Poincar\'{e} section of the Lorenz attractor \cite{Masoller}.
	
	\section{Dynamical Abstract Self-Similar Sets and Chaos} \label{DynConstructOfASSS}

	In this part of the paper, we describe a dynamical determination of abstract self-similar set by utilizing the roles of the domain and the map simultaneously. In other words, a map is used to describe the structure of $ \mathcal{F} $ and the relationships between its subsets. The set constructed by this way, we call it Dynamical Abstract Similarity Set (DASS). We start by considering the triple $ (X, d, \varphi) $ and a compact set $ F \subset X $, where $ d $ is a metric, and $ \varphi: X \to X $ is a map.
	
	Let $ m $ be a fixed natural number and consider the set $ F_0 \subset X $. Denote by $ F^{(1)} $ the preimage of the set $ \varphi(F) \cap F_0 $ under the function $ \varphi $ in $ F $ and assume that there exist disjoint nonempty subsets $ F_i \subset F, \; i=1, 2, ... m $, such that $ \cup_{i=1}^{m} F_i = F^{(1)} $. 
	
	Denote by $ F^{(2)} $ the preimage of the set $ \varphi(F) \cap F_0 $ under $ \varphi^2 $ in $ F^{(1)} $ and assume that there exist disjoint nonempty subsets $ F_{ij} \subset F_i, \; j=1, 2, ... m $, such that $ \cup_{j=1}^{m} F_{ij} = F^{(2)} $. 
	
	Once more, denote by $ F^{(3)} $ the preimage of the set $ \varphi(F) \cap F_0 $ under $ \varphi^3 $ in $ F^{(2)} $ and assume that there exist disjoint nonempty subsets $ F_{ijk} \subset F_{ij}, \; k=1, 2, ... m $, such that $ \cup_{k=1}^{m} F_{ijk} = F^{(3)} $. 
	
	In general, if the sets $ F^{(n-1)} $ are determined, we denote by $ F^{(n)} $ the preimage of the set $ \varphi(F) \cap F_0 $ under $ \varphi^n $ in $ F^{(n-1)} $ and assume that there exist disjoint nonempty subsets $ F_{i_1 i_2 ... i_n} \subset F_{i_1 i_2 ... i_{n-1}}, \; i_n=1, 2, ... m $, such that $ \cup_{i_n=1}^{m} F_{i_1 i_2 ... i_n} = F^{(n)} $. 

	We continue in this procedure, and assume that the following condition is satisfied
	\begin{equation} \label{FDiamCond}
	\max_{i_k=1,2, ..., m} \mathrm{diam}(F_{i_1 i_2 ... i_n}) \to 0 \;\; \text{as} \;\; n \to \infty.
	\end{equation}
	Let us construct a sequence, $ p_n $, of points in $ F $ such that $ p_0 \in F $, $ p_1 \in F_{i_1} $, $ p_2 \in F_{i_1 i_2} $, ... , $ p_n \in F_{i_1 i_2 ... i_n}, \; n=1, 2, ...\, $. It is clear that,
	\[ F \supset F_{i_1} \supset F_{i_1 i_2} \supset ... \supset F_{i_1 i_2 ... i_n} \supset F_{i_1 i_2 ... i_n i_{n+1}} ... , \; i_k=1, 2, ... , m, \; k=1, 2, ... \, . \]
	That is, the sets form a nested sequence. Therefore, due to the compactness of $ F $ and condition (\ref{FDiamCond}), there exists a unique limit point for the sequence $ p_n $. According to the indexes of the nested subsets, we denote the point as $ \mathcal{F}_{i_1 i_2 ... i_n ... } \in F $. Conversely, it can be verified that each point $ p = \mathcal{F}_{i_1 i_2 ... i_n ...} $ admits a corresponding $ p_n $. Based on this, one can justify that the representation of each such point is a unique one. The collection of all such points constitutes the set $ \mathcal{F} $, i.e.,
	\begin{equation*} \label{ASSet}
		\mathcal{F} =  \big\{\mathcal{F}_{i_1 i_2 ... i_n ... } : i_k=1,2, ..., m \},
	\end{equation*}
	and for fixed indexes $ i_1, i_2, ..., i_n $ the subsets of  $ \mathcal{F} $ can be represented by
	\begin{equation*} \label{ASSubSet}
		\mathcal{F}_{i_1 i_2 ... i_n} = \bigcup_{j_k=1,2, ..., m } \mathcal{F}_{i_1 i_2 ... i_n j_1 j_2 ... }.
	\end{equation*}
	Since $ \mathcal{F}_{i_1 i_2 ... i_n} \subset F_{i_1 i_2 ... i_n} $, the condition (\ref{FDiamCond}) implies that the diameter condition (\ref{Diamprop}) is valid for the set $ \mathcal{F} $. Thus, the set $ \mathcal{F} $ is a DASS. Moreover, from the above construction, we see that the map $ \varphi $ satisfies the relations (\ref{MapDefn}) and (\ref{MapSubset}). Therefore, $ \varphi $ is a similarity map and triple $ (\mathcal{F} , d, \varphi) $ is a self similar space.
	
	Now, let us formulate the following condition: For arbitrary $ i_1 i_2 ... i_n $ one can find $ j_1 j_2 ... j_n $ and a positive number $ \varepsilon $ such that
	\begin{equation} \label{FSeperation}
	d \big( F_{i_1 i_2 ... i_n} \, , \, F_{j_1 j_2 ... j_n} \big) \geq \varepsilon.
	\end{equation}
		
	From the construction, it is clear that if the condition (\ref{FSeperation}) holds for the sets $ F_{i_1 i_2 ... i_n} $, then the separation condition (\ref{C2}) is valid for the set $ \mathcal{F} $ with a separation constant $ \varepsilon_0 \geq \varepsilon $. If this is the case, then in view of Theorem \ref{Thm1}, \ref{Thm2} and \ref{Thm3}, the similarity map $ \varphi $ is chaotic in the sense of Poincar\'{e}, Li-Yorke and Devaney.
	
	The approach described above is not only an alternative way of the abstract similarity construction but it can be an essential part of the subject. For instance, it gives us a method for indexing. That is, the similarity map can be simultaneously used to number the subsets of each order depending on the numeration of their images. The following examples illustrate the idea of DASS, indexing and chaos.

	\begin{example} \label{Exm3}
	
	Let $ X=\mathbb{R}^n $ and $ F_0 = F = [0, 1]^n $ is the $ n $-dimensional unit cube. Consider the $ n $-dimensional logistic map $ \varphi = (\varphi_1, \varphi_2, ... \varphi_n): \mathbb{R}^n \to \mathbb{R}^n $ defined by
	\begin{equation}
	\begin{split}
	x_{k+1}^1 &= \varphi_1(x_k^1)= r_1 x_k^1 (1-x_k^1), \\
	x_{k+1}^2 &= \varphi_2(x_k^2)= r_2 x_k^2 (1-x_k^2), \\
	. \\
	. \\
	x_{k+1}^n &= \varphi_n(x_k^n)= r_n x_k^n (1-x_k^n),
	\end{split}
	\label{nDLogistic}
	\end{equation}
	where $ r_i > 4, \; i=1,2, ..., n $ are parameters. Proceeding upon the properties of the logistic map, ETA is applied for the map (\ref{nDLogistic}). We iterate the points in $ F $ under the map (\ref{nDLogistic}) such that in each iteration, we keep only the points whose images do not escape the domain $ F_0 $. The resulting points from the first iteration are belong to the subsets $ F_i, \; i=1, 2, ..., 2^n $. In the second iteration, the non-escaped points are belong to $ 2^{2n} $ subsets and each subset is indexed as $ F_{ij}, \; j=1, 2, ..., 2^n $ such that $ F_{ij} \subset F_i $ and $ \varphi(F_{ij}) = F_j $. Similarly, a subset resulting at the $ k^{th} $ iteration is indexed as $ F_{i_1 i_2 ... i_k} $ such that $ F_{i_1 i_2 ... i_k} \subset F_{i_1 i_2 ... i_{k-1}} $ and $ \varphi(F_{i_1 i_2 ... i_k}) = F_{i_2 i_3 ... i_k} $.
	
	Based on the algorithm, it is clear that the condition (\ref{FDiamCond}) holds. Thus, we describe the points $ \mathcal{F}_{i_1 i_2 i_3 ... } = \lim_{ k \to \infty} F_{i_1 i_2 ... i_k} $, and then the DASS, the self-similar set $ \mathcal{F} $ corresponding to the above algorithm, is defined as the collection of the points $ \mathcal{F}_{i_1 i_2 i_3 ... } $.
	
	For $ r_i, \; i=1,2, ..., n $ larger than 4, the separation condition is guaranteed to be valid for the set $ \mathcal{F} $, and therefore, Theorem \ref{Thm1}, \ref{Thm2} and \ref{Thm3} imply that the similarity map $ \varphi $ is chaotic in the sense of Poincar\'{e}, Li-Yorke and Devaney.
	
	For numerical simulation, let us consider the 2-dimensional system
	
	\begin{equation}
	\begin{split}
	x_{n+1} &= \varphi_1(x_n)= r_1 x_n (1- x_n), \\
	y_{n+1} &= \varphi_2(y_n)= r_2 y_n (1- y_n),		
	\end{split}
	\label{2DLogistic}
	\end{equation}
	with $ r_1=4.2 $ and $ r_2=4.3 $. We fix $ F_0 = F = [0, 1] \times [0, 1] $ and apply ETA for (\ref{2DLogistic}). The first iteration will generate the sets $ F_i, \; i=1, 2, 3, 4 $. In the second iteration, we get the sets $ F_{ij}, \; j=1, 2, 3, 4 $, and so on. Figure \ref{2DLogisticSet} shows the subsets constructed in the first three iterations. The DASS corresponding to the system (\ref{2DLogistic}) is the set resulting from an infinite iteration of this procedure. The set is a sort of Cantor dust which is the Cartesian product of two Cantor sets \cite{Layek}.
	
	\begin{figure}[H]
	\centering
	\subfigure[]{\includegraphics[width = 2.0in]{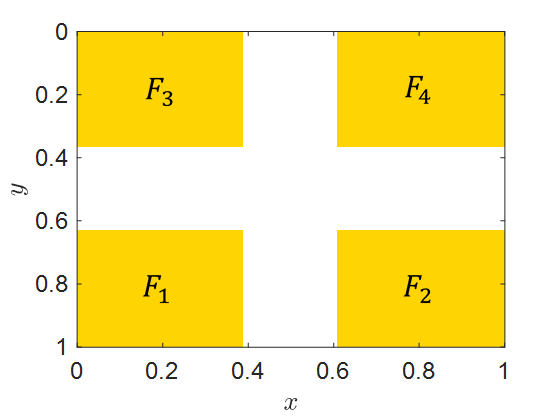}}
	\subfigure[]{\includegraphics[width = 2.0in]{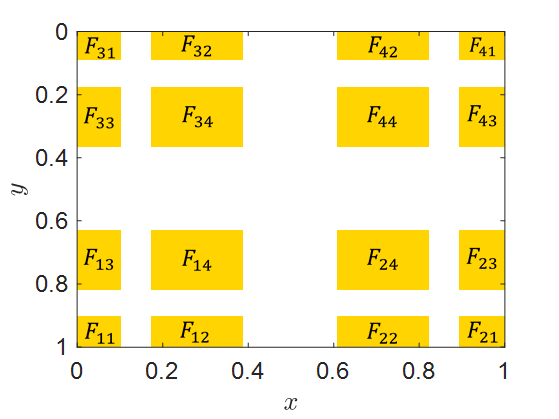}}
	\subfigure[]{\includegraphics[width = 2.0in]{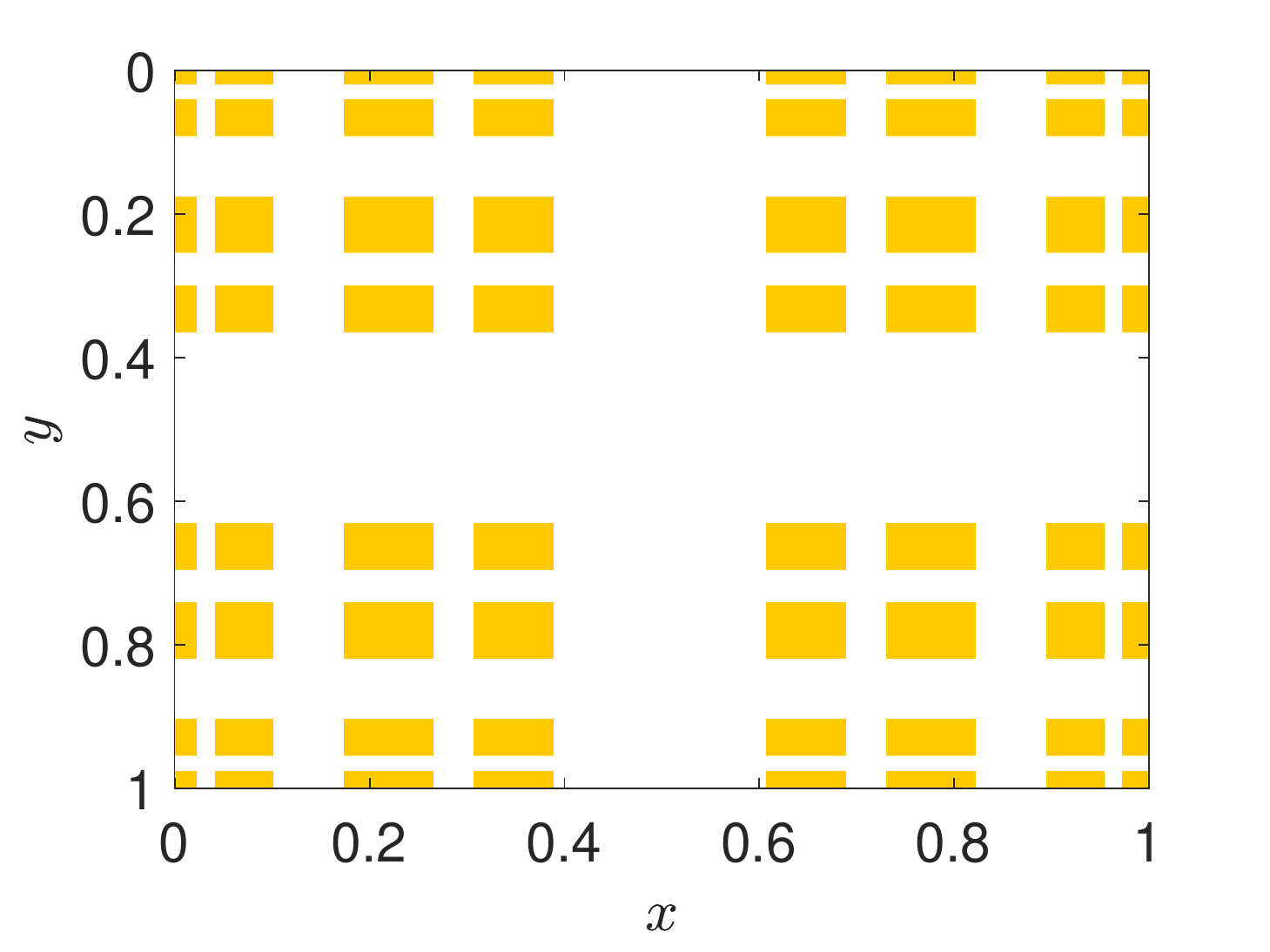}}
	\caption{The construction of abstract self-similar set using the map (\ref{2DLogistic})}
	\label{2DLogisticSet}   				
	\end{figure}

	\end{example}

	\begin{example} \label{Exm4}
	
	Let $ F_0 = F $ denote the initial set $ [0, 1] \times [0, 1] $ and consider the 2-dimensional perturbed logistic map $ \varphi = (\varphi_1, \varphi_2): \mathbb{R}^2 \to \mathbb{R}^2 $ defined by
	\begin{equation}
		\begin{split}
			x_{n+1} &= \varphi_1(x_n, y_n; \mu_1)= r_1 x_n (1- x_n)+ \mu_1 y_n, \\
			y_{n+1} &= \varphi_2(x_n, y_n; \mu_2)= r_1 y_n (1- y_n) + \mu_2 x_n,		
		\end{split}
		\label{2DPertuLogistic}
	\end{equation}
	where $ r_1, r_2, \mu_1 $ and $ \mu_2 $ are parameters. The last term in the left-hand side of the both equations in system (\ref{2DPertuLogistic}) can be considered as a perturbation of a unimodal map. It is known that, if a unimodal map is regular \cite{Zeraoulia,Hunt}, then it is structurally stable, and therefore, any small perturbation does not affect the topological properties of the map \cite{Avila}. Such an inference can be extended for high-dimensional unimodal maps. Numerical simulations can provide an adequate verification of the unimodal properties of the perturbed map.
	
	Similar to Example \ref{Exm3}, we apply ETA to the map (\ref{2DPertuLogistic}). The points that do not escape $ F_0 $ in the first iteration are belong to the subsets $ F_i, \; i=1, 2, 3, 4 $. In the second iteration, the resulting points are belong to the subsets indexed by $ F_{ij}, \; j=1, 2, 3, 4 $ such that $ F_{ij} \subset F_i $ and $ \varphi(F_{ij}) = F_j $. Similarly, a subset resulting at the $ n^{th} $ iteration is indexed as $ F_{i_1 i_2 ... i_n} $ such that $ F_{i_1 i_2 ... i_n} \subset F_{i_1 i_2 ... i_{n-1}} $ and $ \varphi(F_{i_1 i_2 ... i_n}) = F_{i_2 i_3 ... i_n} $. Figure \ref{LogisticSet} shows the subsets constructed in the first three iterations with the parameters $ r_1=4.2, r_2=4.5, a=0.03 $, and $ b=-0.05 $.
	
	\begin{figure}[H]
		\centering
		\subfigure[]{\includegraphics[width = 2.0in]{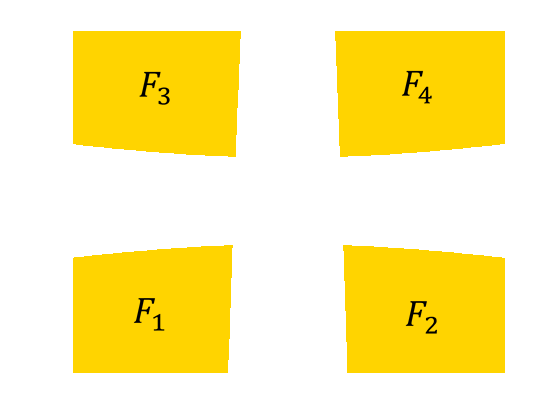}}
		\subfigure[]{\includegraphics[width = 2.0in]{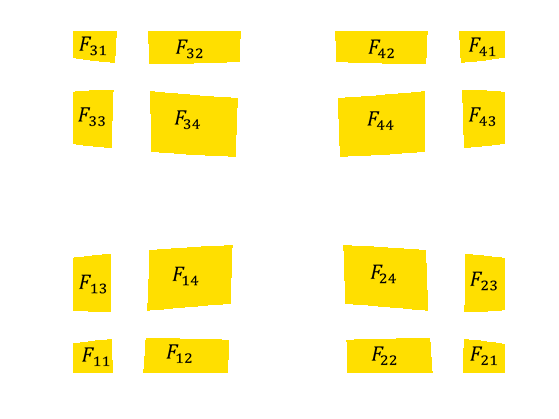}}
		\subfigure[]{\includegraphics[width = 2.0in]{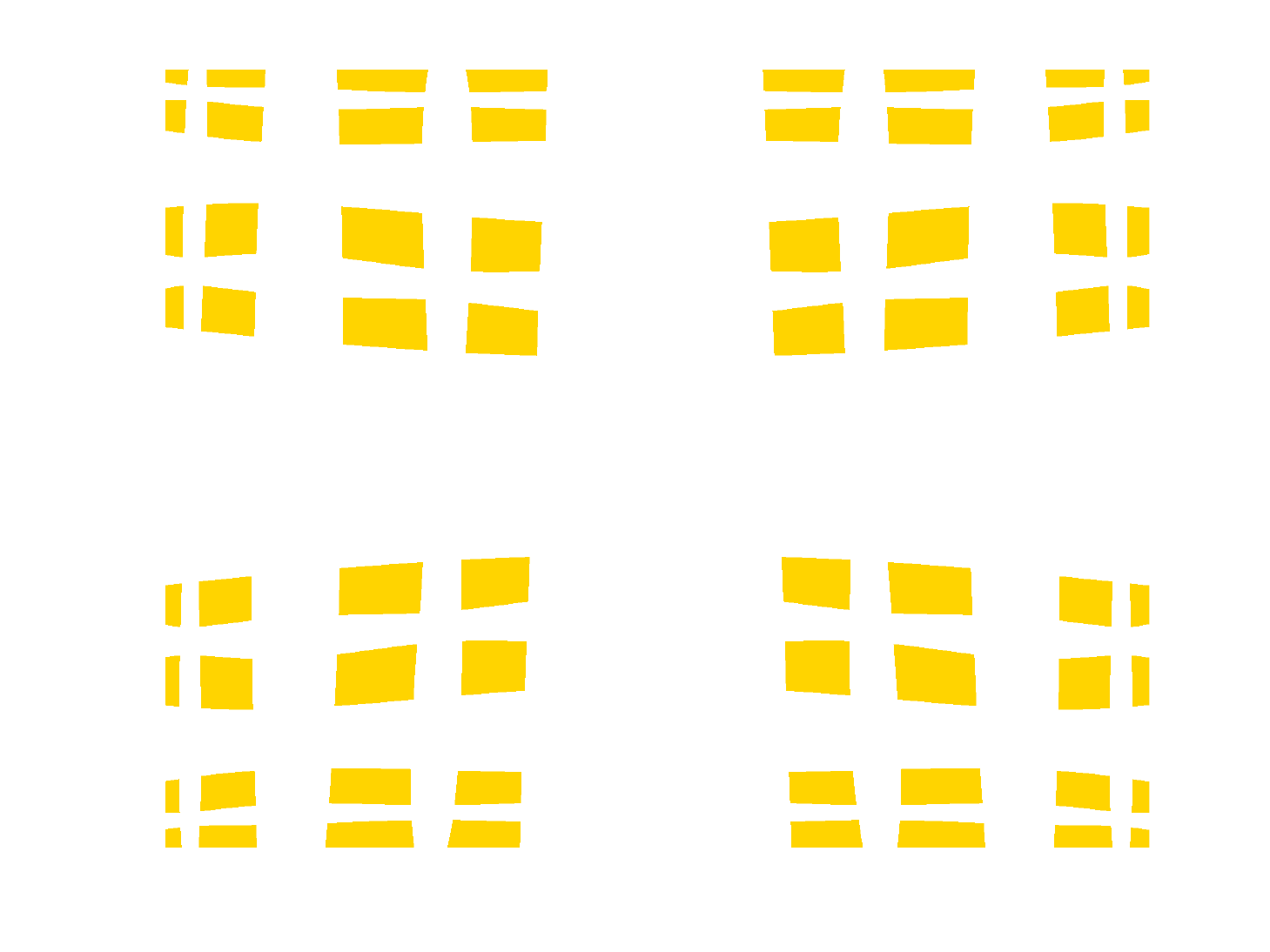}}
		\caption{The first three iterations of DASS construction using the map (\ref{2DPertuLogistic})}
		\label{LogisticSet}   				
	\end{figure}

	Depending on the choice of the coefficients $ r_i $ and relying on the smallness of the coefficients $ \mu_i $, we have that the diameter and separation conditions for abstract similarity and chaos are fulfilled. Moreover, the simulation results confirm that both conditions hold. Therefore, we could say that the similarity map (\ref{2DPertuLogistic}) is chaotic on the self similar-set $ \mathcal{F} $. Figure \ref{LogisticTrajec} depicts the trajectories of some points that approximately belong to the set $ \mathcal{F} $. The irregular behavior of the trajectories reveals the presence of chaos in (\ref{2DPertuLogistic}).  

	\begin{figure}[H]
	\centering
	\subfigure[The trajectory which starts at the point $ (0.044608921784357, 0.287657531506301) $]{\includegraphics[width = 2.5in]{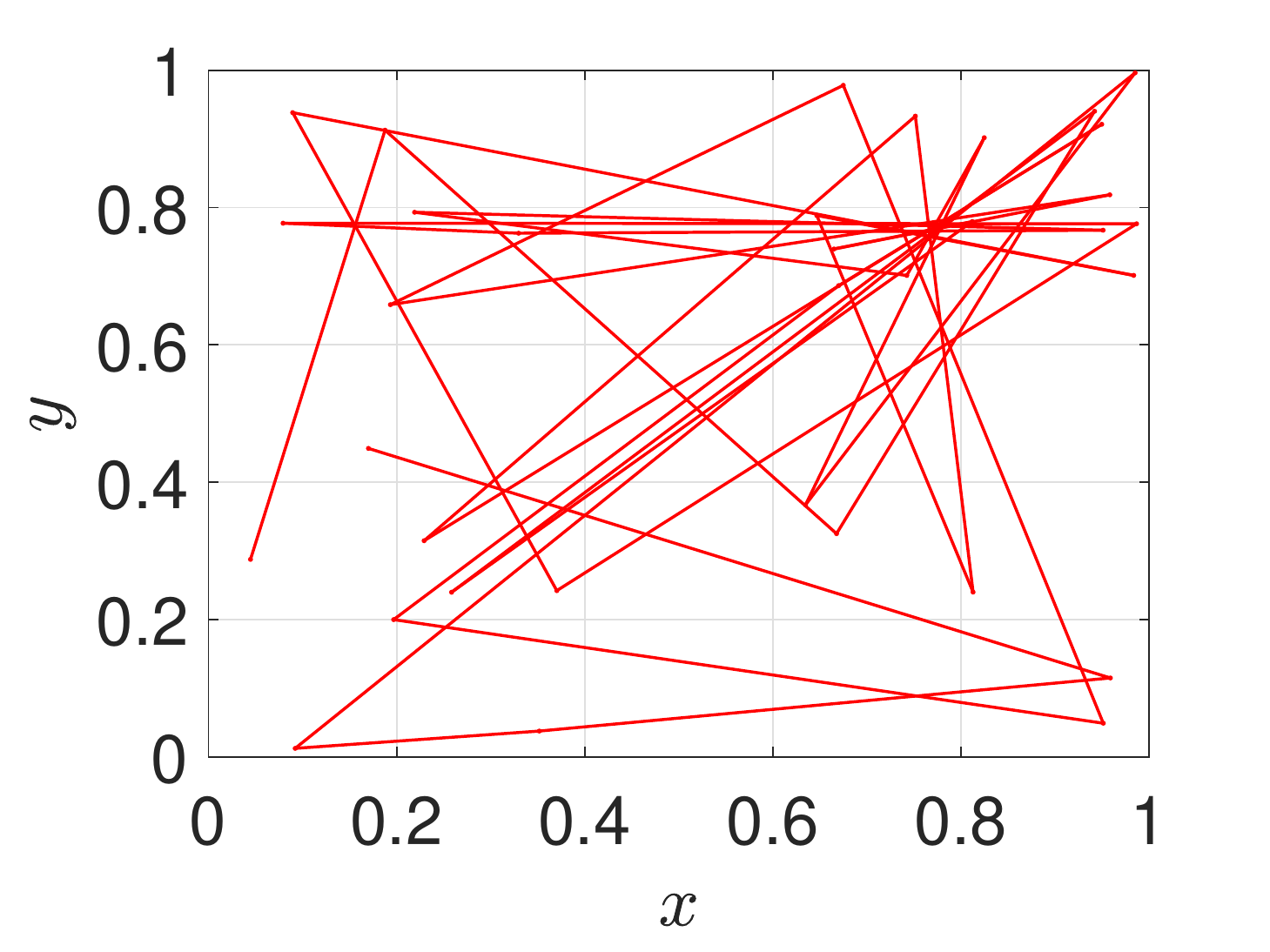}\label{KCChaosa}} \hspace{0cm}
	\subfigure[The trajectory which starts at the point $ (0.910182036407281, 0.329865973194639) $]{\includegraphics[width = 2.5in]{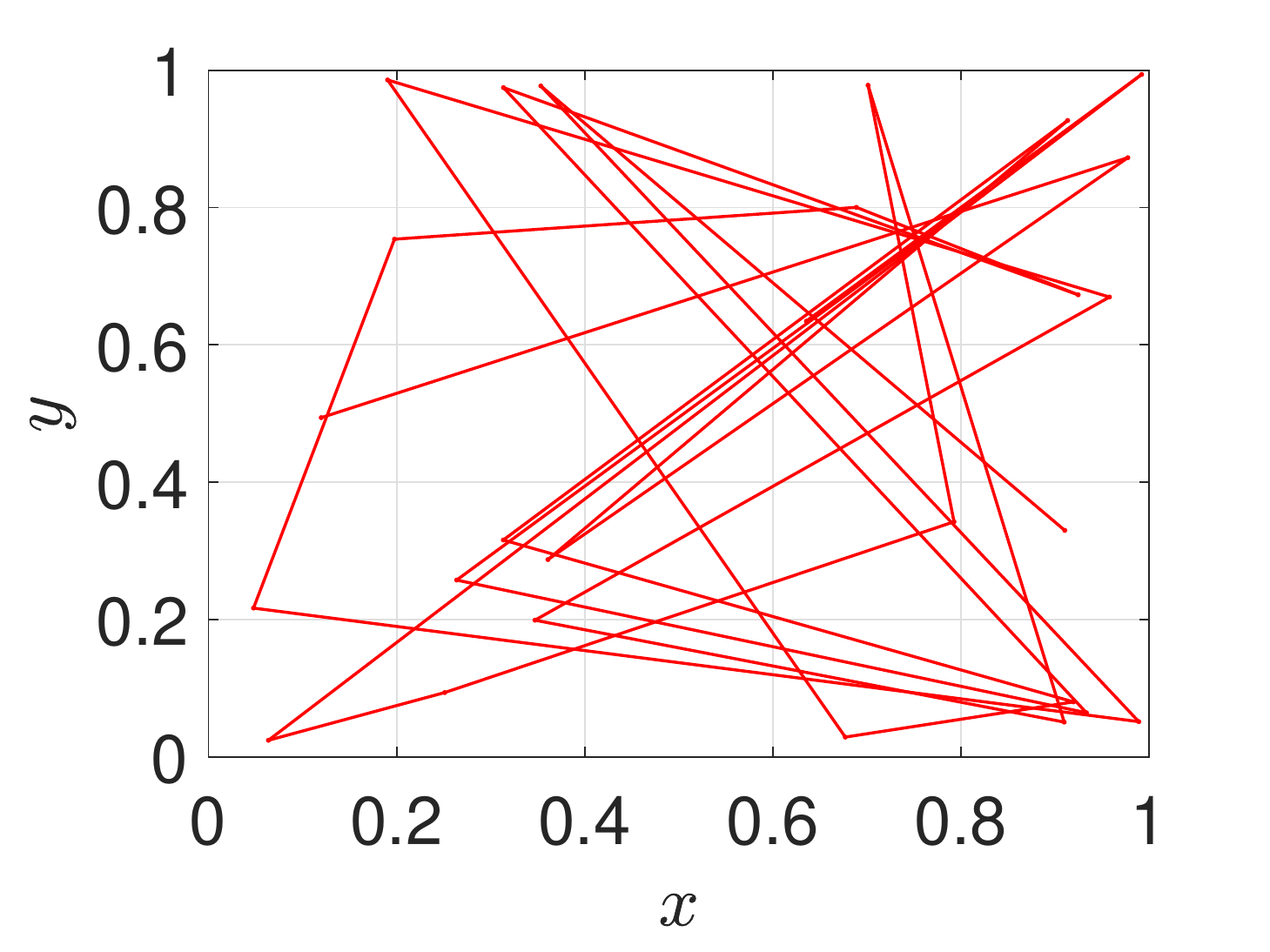}\label{KCChaosa}} \hspace{0cm}
	\caption{The two trajectories of the system (\ref{2DPertuLogistic})}
	\label{LogisticTrajec}
	\end{figure}
	
	\end{example}

	\begin{example} \label{Exm5}
	
	Consider the space $ X=\mathbb{R}^n $ and let $ F $ be a compact set in $ X $ such that it contains an open neighborhood of the $ n $-dimensional unit cube and the set $ F_0 $ to be sufficiently near to the cube.  Consider the n-dimensional perturbed logistic map $ \varphi = (\varphi_1, \varphi_2, ... \varphi_n): \mathbb{R}^n \to \mathbb{R}^n $ which is defined by
	
	\begin{equation}
	\begin{split}
	x_{k+1}^1 &= \varphi_1(x_k^1, x_k^2, ..., x_k^n; \mu_1)= r_1 x_k^1 (1-x_k^1) +\mu_1 \chi_1(x_k^1, x_k^2, ..., x_k^n), \\
	x_{k+1}^2 &= \varphi_2(x_k^1, x_k^2, ..., x_k^n; \mu_2)= r_2 x_k^2 (1-x_k^2) +\mu_2 \chi_2(x_k^1, x_k^2, ..., x_k^n), \\
	. \\
	. \\
	x_{k+1}^n &= \varphi_n(x_k^1, x_k^2, ..., x_k^n; \mu_n)= r_n x_k^n (1-x_k^n) +\mu_n \chi_n(x_k^1, x_k^2, ..., x_k^n),		
	\end{split}
	\label{nDPertuLogistic}
	\end{equation}
	where $ \mu_i, \; i=1,2, ..., n $ are parameters and $ \chi= (\chi_1, \chi_2, ..., \chi_n; \mu_i) $ is a continuous function. Due to the continuity of $ \chi $ and for $ r_i, \; i=1,2, ..., n $ larger than $ 4 $ and sufficiently small $ \mu_i, \; i=1,2, ..., n $, one can show that a DASS can constructed using (\ref{nDPertuLogistic}), and thus, chaotic behavior, in the since of Poincar\'{e}, Li-Yorke and Devaney, for n-dimensional perturbed logistic map would be expected to appear.
	\end{example}

	\section{Conclusion}

	The abstraction of the self-similarity concept is now accomplished. Furthermore, we have shown that the set of symbolic strings satisfies the definition of abstract self-similarity. This example illustrates how the abstraction of a mathematical concept can be significant not only to extract its essence but also to explore more fields where it can be manifested. In addition to equipping the self-similar set with a metric, the similarity map is introduced to define abstract self-similar space. The map is proven to be chaotic in the sense of Poincar\'{e}, Li-Yorke and Devaney. The building of chaos usually begins with a map defined over its domain and then saying about a chaotic attractor that appears as a part of the domain. In our research, we start by describing a chaotic set, and only then introduce a similarity map which admits chaotic dynamics. We utilize infinite sequences to index the points of the domain. The action of the map is not just a shifting in the string space as much as a transforming of the domain points.
	
	Self-similarity is widely spread in nature, but it is usually associated with fractal geometry \cite{Mandelbrot1,Peitgen}. Proceeding from this point, we have shown that the Sierpinski fractals, Koch curve and Cantor set can be associated with abstract self-similarity, and consequently possess chaos. This covers already known fractals constructed through self-similarity and possibly other fractals that generated by escape-time algorithm such as Julia and Mandelbrot sets. The suggested abstract similarity definition can be elaborated through fractal sets defined by fractal dimension, chaotic dynamics development, topological spaces, physics, chemistry, and neural network theories development \cite{Boeing,Zaslavsky,Hata1,Sato}.

\end{document}